\documentclass[10pt]{article}
\usepackage{amsfonts}
\usepackage{multicol}
\usepackage{amsmath}
\usepackage{bm}
\usepackage{pgf}
\usepackage[square, comma, sort&compress, numbers]{natbib}
\usepackage[english]{babel}
\usepackage[utf8]{inputenc}
\usepackage{datetime}
\usepackage{geometry}
\usepackage{array}
\usepackage{xspace}

\usepackage{amsmath,amsfonts,amssymb,amsopn,amsthm}
\theoremstyle{plain}
\usepackage{graphicx} 
\usepackage{epstopdf}
\usepackage{enumerate}

\def\Box{\vcenter{\vbox{\hrule\hbox{\vrule
				\vbox to 8.8pt{\hbox to 10pt{}\vfill}\vrule}\hrule}}}
\usepackage{indentfirst}
\newtheorem{theorem}{Theorem}[section]

\newtheorem{definition}[theorem]{Definition}
\newtheorem{example}[theorem]{Example}

\newtheorem{lemma}[theorem]{Lemma}

\newtheorem{remark}[theorem]{Remark}

\large\normalsize

\begin{document}

%

\title{$m$-weak group MP inverse
}

\author{\\Wanlin Jiang$^{a}$, Jiale Gao$^{*,b}$, Xiangyu Zhang$^{c}$ and Shengxi Zuo$^{d}$\\
\small $^{a}$School of Science and Technology, College of Arts and Science of Hubei Normal University, Huangshi, 435109, China\\
\small $^{b}$Department of Mathematics, Shanghai University, Shanghai 200444, China\\
\small $^{c}$School of Mathematics and Statistics, Hubei Normal University, Huangshi, 435002, China\\
\small $^{d}$University of Warwick, Coventry, CV4 7AL, United Kingdom\\
}

\date{}
\maketitle
\begin{abstract} In this paper, we introduce a new matrix decomposition called the $m$-Core-nilpotent decomposition which is an extension of the Core-nilpotent decomposition. By this new decomposition, we propose a new generalized inverse named the $m$-weak group MP inverse which unifies the DMP-inverse and weak core inverse. Some characterizations, properties and representations of the $m$-weak group MP inverse are presented. In addition, the proposed generalized inverse is applicable to solving a restricted matrix equation.
\end{abstract}

$\mathbf{Keywords:}$ $m$-Core-nilpotent decomposition; $m$-weak group inverse; $m$-weak group MP inverse; weak core inverse; DMP-inverse

$\mathbf{Mathematics}$ $\mathbf{Subject}$ $\mathbf{Classfication:}$ 15A09

\let\thefootnote\relax\footnotetext{$^*$E-mail address: jialegao819@163.com}

\let\thefootnote\relax\footnotetext{This research is supported by NSFC(NO.11961076)}

\section{Introduction}\numberwithin{equation}{section}
Let $\mathbb{Z}^{+}$ and $\mathbb{C}^{m\times n}$ denote the sets of all positive integers and all $m\times n$ complex matrices, respectively. The symbols $r(A)$, $\mathcal{R}(A)$, $\mathcal{N}(A)$ and $A^{*}$ denote the rank, range space, null space and conjugate transpose of $A\in\mathbb{C}^{m\times n}$, respectively. The smallest nonnegative integer $k$ such that $r(A^{k+1})=r(A^{k})$ is called the index of $A\in\mathbb{C}^{n\times n}$ and is denoted by ${\rm Ind}(A)$. If ${\rm Ind}(A)\leq1$, then $A$ is called a core matrix. The notation $\mathbb{C}^{n\times n}_{k}$ denotes the set of all $n\times n$ complex matrices with index $k$ . The symbols $I_{n}$ and 0 stand for the identity matrix of order $n$ and null matrix of an appropriate order, respectively.
The symbols $\mathbb{C}^{\rm CM}_{n}$ and $\mathbb{C}^{\rm EP}_{n}$ represent the subsets of $\mathbb{C}^{n\times n}$ consisting of all core matrices and EP matrices, respectively, i.e.,
\begin{eqnarray} \nonumber
&& \mathbb{C}^{\rm CM}_{n}=\{A|A\in\mathbb{C}^{n\times n}, r(A)=r(A^{2})\},\\  \nonumber
&&\mathbb{C}^{\rm EP}_{n}=\{A|A\in\mathbb{C}^{n\times n},\mathcal{R}(A)= \mathcal{R}(A^{*})\}. \nonumber
\end{eqnarray}

The well-known Moore-Penrose inverse (in short, MP-inverse $A^{\dag}$) \cite{P,WWQ}, Drazin inverse $A^{D}$ \cite{MBM,D,WSP,CW} and group inverse $A^{\#}$ \cite{IE} have been extensively studied for decades since 1920. As an alternative to the group inverse, the core inverse which exists only for the square matrix with the index $k\leq1$ was introduced by Baksalary and Trenkler \cite{BT} in 2010, and a new wave of studying the core inverse and its extensions \cite{PM,MT,BTT,DFAN} was initiated. From 2014 onwards, the core-EP inverse \cite{PM}, DMP-inverse \cite{MT}, BT-inverse \cite{BTT} and weak core inverse \cite{DFAN} were introduced to extend the core inverse to the square matrix with arbitrary index $k$. More characterizations and applications about these generalized inverses can be found in \cite{FZC,DMPS,MM,ZCC,ZC,FLT1,LZZ,HK}. Recently, the $m$-weak core inverse \cite{WHN} unifies the weak core inverse and core-EP inverse.

  Let $\mathbb{C}^{n}$ be the $n$-dimensional complex vector space and let $P_{\mathcal{L},\mathcal{M}}$ be the projector on the space $\mathcal{L}$ along the $\mathcal{M}$ such that $\mathcal{L}\oplus \mathcal{M}=$$\mathbb{C}^{n}$. For $A\in\mathbb{C}^{m\times n}$, $P_{A}$ represents the orthogonal projection onto $\mathcal{R}(A)$, i.e., $P_{A}=P_{\mathcal{R}(A)}=AA^{\dag}$. For $A\in\mathbb{C}^{n\times n}_{k}$ and $m\in\mathbb Z^{+}$, the representations of the core-EP inverse, DMP-inverse, weak group inverse and $m$-weak group inverse, which are denoted by $A^{\textcircled{\emph{\dag}}}$, $A^{D,\dag}$, $A^{\textcircled{w}}$ and $A^{\textcircled{w}_{m}}$, respectively, are as follows \cite{flt3,MT,WC,ZCZ,WL,WLJ,ZCZM}:
\begin{equation}\label{eq.1}
A^{\textcircled{\emph{\dag}}}=(A^{k+1}(A^{k})^{\dag})^{\dag}=A^{D}P_{A^{k}},
\end{equation}
\begin{equation}\label{eq.2}
A^{D,\dag}=A^{D}P_{A},
\end{equation}
\begin{equation}\label{eq.3}
 A^{\textcircled{w}}=(A^{\textcircled{\emph{\dag}}})^{2}A,
\end{equation}
\begin{equation}\label{eq.4}
 A^{\textcircled{w}_{m}}=(A^{\textcircled{\emph{\dag}}})^{m+1}A^{m}.
\end{equation}

\begin{remark}\label{rem1.1}
From \cite{WLJ}, it is known that the $m$-weak group inverse unifies the weak group inverse and Drazin inverse. If $m=1$, then $A^{\textcircled{w}_{1}}$ coincides with $A^{\textcircled{w}}$. If $m\geq k$, then $A^{\textcircled{w}_{m}}$ coincides with $A^{D}.$
\end{remark}

In 2020, Ferreyra et al. \cite{DFAN} introduced the weak core inverse by using the MP-inverse and Drazin inverse. The weak core inverse of $A\in\mathbb{C}^{n\times n}_{k}$, denoted by $A^{\textcircled{w},\dag}$, is the unique matrix $X\in\mathbb{C}^{n\times n}$ satisfying
\begin{equation}\label{eq1}
    XAX=X, AX=CA^{\dag}\  \mbox{and}\  XA=A^{D}C,
\end{equation}
where $C=AA^{\textcircled{w}}A\in\mathbb{C}^{\rm CM}_{n}$. Meanwhile,
\begin{equation}\label{eq.6}
A^{\textcircled{w},\dag}=A^{\textcircled{w}}P_{A}.
 \end{equation}

Soon afterwards, the Ferreyra and Trenkler \cite{WHN} defined the $m$-weak core inverse of $A\in\mathbb{C}^{n\times n}_{k}$ as
\begin{equation}\label{eq.7}
A^{\textcircled{\#}_{m}}=A^{\textcircled{w}_{m}}P_{A^{m}}.
 \end{equation}
From $(\ref{eq.1})$, $(\ref{eq.6})$, $(\ref{eq.7})$ and Remark $\ref{rem1.1}$, we know that if $m=1$, then $A^{\textcircled{\#}_{1}}=A^{\textcircled{w},\dag}.$ If $m\geq k$, then $A^{\textcircled{\#}_{m}}=A^{D}P_{A^{k}}=A^{\textcircled{\emph{\dag}}}.$ Therefore the $m$-weak core inverse extends the notions of the core-EP inverse and weak core inverse.

From $(\ref{eq.2})$ and Remark $\ref{rem1.1}$, we notice that if $A^{\textcircled{w}}$ is replaced by $A^{\textcircled{w}_{m}}$ in $(\ref{eq.6})$, then $A^{\textcircled{w}_{m}}P_{A}$ coincides with $A^{\textcircled{w},\dag}$ and $A^{D,\dag}$ when $m=1$ and $m\geq k-1$, respectively. Therefore, it is natural to investigate $A^{\textcircled{w}_{m}}P_{A}$ .

Motivated by the ideas above, our aim is to define and investigate the new generalized inverse which unifies the DMP-inverse and weak core inverse. The major outcomes of this paper are as follows.

 $(a)$ A new matrix decomposition called the $m$-core-nilpotent decomposition is proposed and it coincides with the core-nilpotent decomposition when $m\geq k-1$. Interestingly, the $m$-core-nilpotent decomposition can be applicable to defining $A^{\textcircled{w}_{m}}P_{A}$, which is named the $m$-weak group MP inverse.

 $(b)$ Some characterizations of the $m$-weak group MP inverse are derived in terms of the range space, null space, rank equalities, matrix equations
 and projectors, respectively. In addition, some expressions of the $m$-weak group MP inverse are given.

 $(c)$ The $m$-weak group MP inverse is used to solve a restricted matrix equation. Meanwhile, a nonsingular bordered matrix constructed by using the null space and range space of the $m$-weak group MP inverse can be applied to the Cramer's rule for the solution of the mentioned restricted matrix equation.

An outline of this paper is as follows. Some symbols and lemmas are given in Section 2. We devote Section 3 to the introduction of the $m$-core-nilpotent decomposition and the $m$-weak group MP inverse. We focus on the characterizations and expressions of the $m$-weak group MP inverse in Section 4 and Section 5, respectively. Finally, the applications of this new generalized inverse in solving the restricted matrix equation are introduced in Section 6.

\section{Preliminaries}
In this part, we show some necessary symbols and results.

For $A\in\mathbb{C}^{m\times n}$ with $r(A)=r$, the matrix $X\in\mathbb{C}^{n\times m}$ is denoted by $A^{(2)}_{\mathcal{T},\mathcal{S}}$ \cite{WWQ}, such that
$$
X=XAX,\ \mathcal{R}(X)=\mathcal{T} \mbox{ and}\ \mathcal{N}(X)=\mathcal{S},
$$
 where $\mathcal{T}$ is the subspaces of $\mathbb{C}^{n}$ of dimensions $s\ (\leq r)$ and $\mathcal{S}$ is the subspaces of $\mathbb{C}^{m}$ of dimensions $m-s$.

\begin{lemma}{\rm\cite{flt3}}\label{l2.3}
Let $A\in\mathbb{C}^{n\times n}_{k}$. Then the following statements are valid:

$(a)$ $A^{D}=A^{(2)}_{\mathcal{R}{(A^{k})},\mathcal{N}{(A^{k})}};$

$(b)$ $AA^{D}=A^{D}A=P_{\mathcal{R}{(A^{k})},\mathcal{N}{(A^{k})}}.$

\end{lemma}

\begin{lemma}{\rm\cite{WLJ}}\label{l2.4}
Let $A\in\mathbb{C}^{n\times n}_{k}$ and $m\in\mathbb{Z^{+}}$. Then the following statements are valid:

$(a)$ $A^{\textcircled{w}_{m}}=A^{(2)}_{\mathcal{R}{(A^{k})},\mathcal{N}{((A^{k})^{*}A^{m})}};$

$(b)$ $AA^{\textcircled{w}_{m}}=P_{\mathcal{R}(A^{k}),\mathcal{N}((A^{k})^{*}A^{m})};$

$(c)$ $A^{\textcircled{w}_{m}}A=P_{\mathcal{R}(A^{k}),\mathcal{N}((A^{k})^{*}A^{m+1})}.$
\end{lemma}

\begin{lemma}{\rm{\cite{MBM}}} \label{l2.2222}
Let $A\in\mathbb{C}^{n\times n}_{k}$. Then $A$ can be uniquely expressed as $A=\tilde{{A}_{1}}+\tilde{{A}_{2}}$, where
\begin{eqnarray} \nonumber
&& (a)\ \tilde{{A}_{1}}\in\mathbb{C}^{\rm CM}_{n};\\  \nonumber
&& (b)\ (\tilde{{A}_{2}})^{k}=0;\\  \nonumber
&& (c)\ \tilde{{A}_{2}}\tilde{{A}_{1}}=0;\\  \nonumber
&& (d)\ \tilde{{A}_{1}}\tilde{{A}_{2}}=0.  \nonumber
\end{eqnarray}
 The above $A=\tilde{{A}_{1}}+\tilde{{A}_{2}}$ is called the core-nilpotent decomposition of $A$, in which case $\tilde{{A}_{1}}$ and $\tilde{{A}_{2}}$ are called the core part and nilpotent part in the core-nilpotent decomposition of $A$, respectively.

\end{lemma}

\begin{lemma}{\rm{\cite{W}}} \label{l2.33}
Let $A\in\mathbb{C}^{n\times n}_{k}$. Then there exists a unitary matrix $U\in\mathbb{C}^{n\times n}$ such that
\begin{equation}\label{eq1.1}
A=\hat{{A}_{1}}+\hat{{A}_{2}}=U \left[\begin{array}{cc}
T & S \\
0 & N \\
\end{array}
\right] U^{*},
\end{equation}
\begin{equation}\label{e2.2}
\hat{A_{1}}=U \left[\begin{array}{cc}
T & S \\
0 & 0 \\
\end{array}
\right] U^{*}, \ \ \
\hat{A_{2}}=U \left[\begin{array}{cc}
0 & 0 \\
0 & N \\
\end{array}
\right] U^{*},
\end{equation}
where $T\in\mathbb{C}^{t\times t}$ is nonsingular with $t=r(T)=r(A^{k})$ and $N$ is nilpotent of index $k$. The representation $(\ref{eq1.1})$ is called the core-EP decomposition of $A$ and is unique, in which case $\hat{A_{1}}$ and $\hat{A_{2}}$ are called the core part and nilpotent part in the core-EP decomposition of $A$, respectively.
\end{lemma}

\begin{lemma}{\rm\cite{WLJ, DD}}
Let $A\in\mathbb{C}^{n\times n}_{k}$ be of the form $(\ref{eq1.1})$ and $m\in\mathbb{Z^{+}}$. Let $\triangle=(TT^{*}+S(I_{n-t}-N^{\dag}N)S^{*})^{-1}
$ and $T_{m}=\sum\limits_{j=0}^{m-1}T^{j}SN^{m-1-j}$. Then
\begin{equation}\label{e2.8}
A^{\dag}=U \left[\begin{array}{cc}
T^{*}\triangle & -T^{*}\triangle SN^{\dag} \\
(I_{n-t}-N^{\dag}N)S^{*}\triangle & N^{\dag}-(I_{n-t}-N^{\dag}N)S^{*}\triangle SN^{\dag} \\
\end{array}
\right] U^{*},
\end{equation}
\begin{equation}
P_{A}=U \left[\begin{array}{cc}\label{e2.9}
I_{t} & 0 \\
0 & NN^{\dag} \\
\end{array}
\right] U^{*},
\end{equation}

\begin{equation}\label{e2.11}
A^{k}=U\left[\begin{array}{cc}
T^{k} & T_{k}\\
0 & 0 \\
\end{array}
\right]
U^{*},
\end{equation}
\begin{equation}\label{e2.12}
A^{m}=U\left[\begin{array}{cc}
T^{m} & T_{m}\\
0 & N^{m} \\
\end{array}
\right]
U^{*},
\end{equation}

\begin{equation}\label{e.mwg}
A^{\textcircled{w}_{m}}=U \left[\begin{array}{cc}
T^{-1} & (T^{m+1})^{-1}T_{m} \\
0 & 0 \\
\end{array}
\right] U^{*}.\
\end{equation}
\end{lemma}

\section{The $m$-core-nilpotent decomposition and $m$-weak group MP inverse}

Matrix decompositions are powerful tools for studying linear systems, numerical and symbolic computations
of generalized inverses \cite{WWQ,WSP,MBM,CW}. Furthermore, some special matrices in matrix decompositions are closely related to the generalized inverses. For example, the core matrices $\tilde{{A}_{1}}=AA^{D}A$ and $\hat{A_{1}}=AA^{\textcircled{\emph{\dag}}}A$ are the core parts in the core-nilpotent decomposition \cite{MBM,MMAS} and core-EP decomposition \rm{\cite{W}} of $A\in\mathbb{C}^{n\times n}_{k}$, respectively. By $A_{1}=AA^{{\textcircled{w}}_{m}}A\in\mathbb{C}^{\rm CM}_{n}$, in the next theorem we introduce a new matrix decomposition where $A_{1}$
is the core part in this decomposition of $A$.
\begin{theorem}\label{thm3.1}
 Let $A\in\mathbb{C}^{n\times n}_{k}$ and $m\in\mathbb Z^{+}$. Then $A$ can be uniquely expressed as $A=A_{1}+A_{2}$, where
\begin{eqnarray} \nonumber
&& (a)\ A_{1}\in\mathbb{C}^{\rm CM}_{n};\\  \nonumber
&& (b)\ (A_{2})^{k}=0;\\  \nonumber
&& (c)\ A_{2}A_{1}=0;\\  \nonumber
&& (d)\ (A^{k})^{*}A^{m}A_{2}=0.  \nonumber
\end{eqnarray}
 The above $A=A_{1}+A_{2}$ is called the $m$-core-nilpotent decomposition of $A$, in which case $A_{1}$ and $A_{2}$ are called the core part and nilpotent part in the $m$-core-nilpotent decomposition of $A$, respectively.
\begin{proof}
Let $A$ be of the form $(\ref{eq1.1})$. Denote $A_{1}=AA^{{\textcircled{w}}_{m}}A$ and $A_{2}=A-A_{1}.$ By $(\ref{e.mwg})$, we get
$$A_{1}= U\left[\begin{array}{cc}
T & T^{-m}T_{m+1}\\
0 & 0 \\
\end{array}
\right]U^{*}\ \ {\rm and} \ \ A_{2}= U\left[\begin{array}{cc}
0 & -T^{-m}T_{m}N\\
0 & N \\
\end{array}
\right]U^{*}.
 $$
Then it can easily be verified that $A_{1}$ and $A_{2}$ satisfy the conditions $(a)-(d)$.

For uniqueness, let $A=B_{1}+B_{2}$ be another $m$-core-nilpotent decomposition of $A$. Since $B_{2}B_{1}=0$, by \cite[Theorem 2]{MPKSV}, we obtain $\mathcal{R}(B_{1})\subseteq \mathcal{R}(A^{D})=\mathcal{R}(A^{k})$. By $(A^{k})^{*}A^{m}B_{2}=0$, $\mathcal{R}(B_{1})\subseteq \mathcal{R}(A^{k})$  and Lemma $\ref{l2.4}$ $(b)$, it follows that
 $$B_{1}-A_{1}=B_{1}-P_{\mathcal{R}(A^{k}),\mathcal{N}((A^{k})^{*}A^{m})}A=B_{1}-P_{\mathcal{R}(A^{k}),\mathcal{N}((A^{k})^{*}A^{m})}(B_{1}+B_{2})=B_{1}-B_{1}=0.$$
\end{proof}
\end{theorem}

Furthermore, according to the above proof, we get directly another characterization of $m$-core-nilpotent decomposition of $A\in\mathbb{C}^{n\times n}_{k}$ by using $A_{1}=AA^{{\textcircled{w}}_{m}}A$.
\begin{theorem}\label{def3.1}
Let $A\in\mathbb{C}^{n\times n}_{k}$ and $m\in\mathbb Z^{+}$ and let $A=A_{1}+A_{2}$ be the $m$-core-nilpotent decomposition of $A$. Then
$$A_{1}=AA^{{\textcircled{w}}_{m}}A \mbox{ and}\ A_{2}=A-AA^{{\textcircled{w}}_{m}}A.$$
If $A$ is given by $(\ref{eq1.1})$, then
\begin{equation}\label{eq1.212}
A_{1}=U \left[\begin{array}{cc}
T & T^{-m}T_{m+1}\\
0 & 0 \\
\end{array}
\right] U^{*},
\end{equation}and
$$A_{2}= U\left[\begin{array}{cc}
0 & -T^{-m}T_{m}N\\
0 & N \\
\end{array}
\right]U^{*}.$$
\end{theorem}

\begin{remark}\label{rem3.33}
 If $m\geq k-1$, then $A_{1}$ coincides with $\tilde{{A}_{1}}=AA^{D}A$ which is the core part in the core-nilpotent decomposition. According to the uniqueness of the core-nilpotent decomposition, it is clear that the $m$-core-nilpotent decomposition coincides with the core-nilpotent decomposition when $m\geq k-1$. Therefore, this new decomposition is a generalization of the core-nilpotent decomposition. If $m=1$, then $A_{1}$ coincides with $C=AA^{{\textcircled{w}}}A$ \cite{DFAN}.
\end{remark}

Note that the above $C$ can be used to define the weak core inverse of $A$ in $(\ref{eq1})$, we introduce similarly a new generalized inverse by using $A_{1}$, $A^{D}$, $A^{\dag}$ and matrix equations.
\begin{theorem}\label{thm3.4}
Let $A\in\mathbb{C}^{n\times n}_{k}$ and $X\in\mathbb{C}^{n\times n}$. Let $A_{1}$ be as in Theorem $\ref{def3.1}$. Then
\begin{equation}\label{eq2.1}
XAX=X, \ \ \  AX=A_{1}A^{\dag} \ \ and\ \ XA=A^{D}A_{1},
\end{equation} is consistent and the unique solution is
\begin{equation}\label{th3.1d}
X=A^{D}A_{1}A^{\dag}=A^{D}AA^{{\textcircled{w}}_{m}}AA^{\dag}.
\end{equation}

\begin{proof}
Let $X=A^{D}A_{1}A^{\dag}$. A simple calculation leads to $XA=A^{D}AA^{{\textcircled{w}}_{m}}AA^{\dag}A=A^{D}A_{1}.$ By Lemma $\ref{l2.3}$ $(b)$, we get that $AX=AA^{D}AA^{{\textcircled{w}}_{m}}AA^{\dag}=A_{1}A^{\dag}$.
By Lemma $\ref{l2.3}$ $(b)$ and Lemma $\ref{l2.4}$ $(b)$, it follows that $$XAX=A^{D}AA^{{\textcircled{w}}_{m}}AA^{\dag}AA^{D}AA^{{\textcircled{w}}_{m}}AA^{\dag}=A^{D}AA^{{\textcircled{w}}_{m}}AA^{D}AA^{{\textcircled{w}}_{m}}AA^{\dag}=A^{D}A_{1}A^{\dag}.$$

For the uniqueness, we assume that $X_{1}$ and $X_{2}$ satisfy the system $(\ref{eq2.1})$. Then we obtain $AX_{1}=A_{1}A^{\dag}=AX_{2}$ and $X_{1}A=A^{D}A_{1}=X_{2}A$, which implies that $X_{2}=X_{2}AX_{2}=X_{2}AX_{1}=X_{2}AX_{1}=X_{1}AX_{1}=X_{1}.$
\end{proof}
\end{theorem}

\begin{definition}\label{def3.5}
Let $A\in\mathbb{C}^{n\times n}_{k}$ and $m\in\mathbb Z^{+}$. The $m$-weak group MP inverse (in short, $m$-WGMP inverse) of $A$, denoted as $A^{{\textcircled{w}}_{m}, \dag}$, is defined to be the unique solution to $(\ref{eq2.1})$.
\end{definition}

 Next, we provide another expression of the $m$-WGMP inverse and its canonical form by using the core-EP decomposition.

\begin{theorem}\label{th3.7} Let $A\in\mathbb{C}^{n\times n}_{k}$ and $m\in\mathbb Z^{+}$. Then
\begin{equation}\label{th3,7}
A^{{\textcircled{w}}_{m}, \dag}=A^{{\textcircled{w}}_{m}}P_{A}.
\end{equation}
If $A$ is given by $(\ref{eq1.1})$, then \begin{equation}\label{311}
A^{{\textcircled{w}}_{m}, \dag}=U \left[\begin{array}{cc}
T^{-1} & T^{-(m+1)}T_{m}NN^{\dag} \\
0 & 0 \\
\end{array}
\right] U^{*}.
\end{equation}
\begin{proof}
According to the Theorem $\ref{def3.1}$, Lemma $\ref{l2.3}$ $(b)$ and Lemma $\ref{l2.4}$ $(a)$, it follows that
\begin{center}
$A^{{\textcircled{w}}_{m}, \dag}=A^{D}A_{1}A^{\dag}=A^{D}AA^{{\textcircled{w}}_{m}}AA^{\dag}=A^{{\textcircled{w}}_{m}}AA^{\dag}=A^{{\textcircled{w}}_{m}}P_{A}.$
\end{center}
 By using $(\ref{e2.9})$, $(\ref{e.mwg})$ and $(\ref{th3,7})$, $A^{{\textcircled{w}}_{m}, \dag}$ in $(\ref{311})$ can be directly verified.
\end{proof}
\end{theorem}

\begin{remark}\label{rem3.7}
From $(\ref{eq.2})$, $(\ref{eq.6})$, $(\ref{th3,7})$ and Remark $\ref{rem1.1}$, it is easy to derive that

 $(a)$ If $m=1$, then $A^{{\textcircled{w}}_{1}, \dag}=A^{{\textcircled{w}}, \dag}$;

 $(b)$ If $m\geq k-1$, then $A^{{\textcircled{w}}_{m}, \dag}=A^{D, \dag}$;

 $(c)$ If $A\in\mathbb{C}^{\rm EP}_{n}$, then $A^{{\textcircled{w}}_{m}, \dag}=A^{{\textcircled{w}}_{m}}$.\\
Hence, the $m$-WGMP inverse actually unifies the weak core inverse and DMP-inverse.
\end{remark}

The following example verifies that the $m$-WGMP inverse is indeed a new generalized inverse.
\begin{example}
Let\\
$$A=\left[\begin{array}{cc}
I_{4} & I_{4}  \\
0 & N\\
\end{array}
\right], where\  N=\left[\begin{array}{cccc}
0 & 1 & 0  & 1 \\
0 & 0 & 1  & 1 \\
0 & 0 & 0  & 1 \\
0 & 0 & 0  & 0\\
\end{array}
\right].$$ \\
It can be checked that ${\rm Ind}(A)=4$. By computations, we get that:

$$A^{\dag}=\left[\begin{array}{cc}
H_{1} & -N^{\dag} \\
I-H_{1} & N^{\dag} \\
\end{array}
\right], \ \ \
A^{\textcircled{\dag}}=\left[\begin{array}{cc}
I_{4} & 0 \\
0 & 0\\
\end{array}
\right],  \ \ A^{D}=\left[\begin{array}{cc}
 I_{4} & H_{2} \\
0 & 0 \\
\end{array}
\right],$$

$$ A^{D,\dag}=\left[\begin{array}{cc}
I_{4} & H_{3} \\
0 & 0\\
\end{array}
\right], \;\ A^{\textcircled{w}}=\left[\begin{array}{cc}
I_{4} & I_{4} \\
0 & 0\\
\end{array}
\right], \ \ \ A^{\textcircled{w},\dag}=\left[\begin{array}{cc}
 I_{4}  &   H_{4} \\
 0  &  0 \\
\end{array}
\right],$$

$$ A^{\textcircled{w}_{2}}=\left[\begin{array}{cc}
I_{4} & H_{5} \\
0 & 0\\
\end{array}
\right], \; \ \ \ A^{\textcircled{w}_{2},\dag}=\left[\begin{array}{cc}
 I_{4}  &   H_{6} \\
 0  &  0 \\
\end{array}
\right] and \ A^{\textcircled{\#}_{2}}=\left[\begin{array}{cc}
 I_{4}  &   H_{7} \\
 0  &  0 \\
\end{array}
\right],$$
where $$H_{1}=\left[\begin{array}{cccc}
\frac{1}{2} & 0& 0& 0  \\
0 & 1& 0 & 0 \\
0 & 0& 1& 0  \\
0 & 0& 0 & 1 \\
\end{array}
\right], \;\ H_{2}=\left[\begin{array}{cccc}
1 & 1& 1&3 \\
0 & 1& 1& 2\\
0 & 0& 1 &1\\
0 & 0& 0 &1\\
\end{array}
\right], \;\ H_{3}=\left[\begin{array}{cccc}
1 & 1& 1& 0 \\
0 & 1& 1& 0 \\
0 & 0& 1& 0 \\
0 & 0& 0& 0 \\
\end{array}
\right], \;\ H_{4}=\left[\begin{array}{cccc}
1 & 0& 0 & 0\\
0 & 1& 0& 0 \\
0 & 0& 1& 0 \\
0 & 0& 0& 0 \\
\end{array}
\right], $$

$$\\ H_{5}=\left[\begin{array}{cccc}
1 & 1& 0 & 1\\
0 & 1& 1& 1 \\
0 & 0& 1& 1 \\
0 & 0& 0& 1 \\
\end{array}
\right], \;\ H_{6}=\left[\begin{array}{cccc}
1 & 1 &  0 & 0 \\
0 & 1& 1& 0 \\
0 & 0& 1& 0 \\
0 & 0& 0& 0 \\
\end{array}
\right],  \;\ H_{7}=\left[\begin{array}{cccc}
1 & 1 &  0 & 0 \\
0 & 1& 0& 0 \\
0 & 0& 0& 0 \\
0 & 0& 0& 0 \\
\end{array}
\right] and \ N^{\dag}=\left[\begin{array}{cccc}
0 & 0& 0& 0 \\
1 & 0& -1&0 \\
0 &1& -1 & 0\\
0 & 0& 1& 0\\
\end{array}
\right].$$

\end{example}

Some basic properties of the $m$-WGMP inverse including the rank, range space, null space and projectors are given, which play an important role in the characterizations, expressions and applications of the $m$-WGMP inverse in the rest of the sections.

\begin{theorem}\label{th3.10}
Let $A\in\mathbb{C}^{n\times n}_{k}$ and $m\in\mathbb Z^{+}$. Then the following assertions hold:
\begin{itemize}
\item[(a)] $r(A^{{\textcircled{w}}_{m},\dag})=r(A^{k});$
\item[(b)] $\mathcal{R}(A^{{\textcircled{w}}_{m}, \dag})=\mathcal{R}(A^{k})$, $\mathcal{N}(A^{{\textcircled{w}}_{m}, \dag})=\mathcal{N}((A^{k})^{*}A^{m+1}A^{\dag});$
\item[(c)] $A^{\textcircled{w}_{m},\dag}=A^{(2)}_{\mathcal{R}(A^{k}),\mathcal{N}((A^{k})^{*}A^{m+1}A^{\dag})};$
\item[(d)] $AA^{{\textcircled{w}}_{m},\dag}=P_{\mathcal{R}(A^{k}),\mathcal{N}((A^{k})^{*}A^{m+1}A^{\dag})};$
\item[(e)] $A^{{\textcircled{w}}_{m},\dag}A=P_{\mathcal{R}(A^{k}),\mathcal{N}((A^{k})^{*}A^{m+1})}$.
\end{itemize}
\begin{proof}
$(a)$. It is evident from $(\ref{e2.11})$ and $(\ref{311})$.

$(b)$. Since $\mathcal{R}(A^{{\textcircled{w}}_{m},\dag})=\mathcal{R}(A^{{\textcircled{w}}_{m}}AA^{\dag})\subseteq \mathcal{R}(A^{{\textcircled{w}}_{m}}A)=\mathcal{R}(A^{k})$ and since $r(A^{\textcircled{w}_{m},\dag})=r(A^{k})$, we get $\mathcal{R}(A^{{\textcircled{w}}_{m},\dag})=\mathcal{R}(A^{k}).$
From Lemma $\ref{l2.4}$ $(a)$, we get $\mathcal{N}(A^{{\textcircled{w}}_{m}})=\mathcal{N}((A^{k})^{*}A^{m}).$
 If $x\in \mathcal{N}(A^{{\textcircled{w}}_{m},\dag})=\mathcal{N}(A^{{\textcircled{w}}_{m}}AA^{\dag})$, we get that $AA^{\dag}x \in \mathcal{N}((A^{k})^{*}A^{m})$, and later $\mathcal{N}(A^{\textcircled{w}_{m},\dag})\subseteq\mathcal{N}((A^{k})^{*}A^{m+1}A^{\dag}).$
By $(\ref{e2.9})$, $(\ref{e2.11})$, $(\ref{e2.12})$ and $(\ref{311})$, it can be verified that $r((A^{k})^{*}A^{m+1}A^{\dag})=t =r(A^{\textcircled{w}_{m},\dag})$, which implies that $\mathcal{N}(A^{\textcircled{w}_{m},\dag})=\mathcal{N}((A^{k})^{*}A^{m+1}A^{\dag}).$

$(c)$. This directly follows by Theorem $\ref{thm3.4}$ and the assertion $(b)$ .

$(d)$. By the assertion $(c)$ , it is clear that $AA^{{\textcircled{w}}_{m}, \dag}$ and $A^{{\textcircled{w}}_{m},\dag}A$ are idempotent matrices .
 Since $\mathcal{R}(AA^{{\textcircled{w}}_{m}, \dag})=A\mathcal{R}(A^{{\textcircled{w}}_{m}, \dag})=A\mathcal{R}(A^{k})=\mathcal{R}(A^{k+1})=\mathcal{R}(A^{k})$ and since
 $\mathcal{N}(AA^{{\textcircled{w}}_{m}, \dag})=\mathcal{N}(A^{{\textcircled{w}}_{m}, \dag})=\mathcal{N}((A^{k})^{*}A^{m+1}A^{\dag})$ from Theorem $\ref{th3.10}$ $(b)$, we get $AA^{{\textcircled{w}}_{m},\dag}=P_{\mathcal{R}(A^{k}),\mathcal{N}((A^{k})^{*}A^{m+1}A^{\dag})}.$

$(e)$. The proof follows similarly as for the part $(d)$.
\end{proof}
\end{theorem}

\section{Characterizations of the $m$-WGMP inverse}

 In this section, we focus on the characterizations of the $m$-WGMP inverse by using the range space, null
space, rank equalities, matrix equations and projectors.

\begin{theorem}\label{th4.1}
Let $A\in\mathbb{C}^{n\times n}_{k}$ and $m\in\mathbb Z^{+}$. Let $A_{1}$ be as in Theorem $\ref{def3.1}$. Then the following conditions are equivalent:
\begin{itemize}
\item[(a)] $X=A^{\textcircled{w}_{m},\dag}$;
\item[(b)] $\mathcal{R}(X)=\mathcal{R}(A^{k})$ and $AX=A_{1}A^{\dag}$;
\item[(c)] $\mathcal{R}(X)=\mathcal{R}(A^{k})$ and $A^{*}AX=A^{*}A_{1}A^{\dag}$;
\item[(d)] $\mathcal{R}(X)=\mathcal{R}(A^{k})$ and $A^{k}X=A_{1}^{k}A^{\dag}$;
\item[(e)] $\mathcal{N}(X)=\mathcal{N}((A^{k})^{*}A^{m+1}A^{\dag})$ and $XA=A^{{\textcircled{w}}_{m}}A$;
\item[(f)] $\mathcal{N}(X)=\mathcal{N}((A^{k})^{*}A^{m+1}A^{\dag})$ and $XAA^{*}=A^{{\textcircled{w}}_{m}}AA^{*}$.
\end{itemize}
\begin{proof}

 $(a)\Rightarrow(b)$. It is a direct consequence of Theorem \ref{th3.7} and Theorem \ref{th3.10} $(b)$.

$(b)\Rightarrow(c)$.  Premultiplying $AX=A_{1}A^{\dag}$ by $A^{*}$, we have
$$A^{*}AX=A^{*}A_{1}A^{\dag}.$$

$(c)\Rightarrow(d)$. Premultiplying $A^{*}AX=A^{*}A_{1}A^{\dag}$ by $A^{k-1}$$(A^{\dag})^{*}$, we obtain $A^{k}X=A^{k}A^{{\textcircled{w}}_{m}}AA^{\dag}.$
   By Lemma $\ref{l2.4}$ $(b)$ and $(c)$, we get $A_{1}^{2}=AA^{{\textcircled{w}}_{m}}AAA^{{\textcircled{w}}_{m}}A=A^{2}A^{{\textcircled{w}}_{m}}A$, and later \begin{equation}\label{eq...}
   A_{1}^{k}=A^{k}A^{{\textcircled{w}}_{m}}A.
   \end{equation}
    Hence $A^{k}A^{{\textcircled{w}}_{m}}AA^{\dag}=A_{1}^{k}A^{\dag},$ which implies
$A^{k}X=A_{1}^{k}A^{\dag}.$

$(d)\Rightarrow(e)$.  Premultiplying $A^{k}X=A_{1}^{k}A^{\dag}$ by $(A^{{\textcircled{w}}_{m}})^{k}$, we have
$$(A^{{\textcircled{w}}_{m}})^{k}A^{k}X=(A^{{\textcircled{w}}_{m}})^{k}A_{1}^{k}A^{\dag}.$$
By $\mathcal{R}(X)=\mathcal{R}(A^{k})$ and Lemma $\ref{l2.4}$ $(a)$ and $(c)$, it follows that
$$(A^{{\textcircled{w}}_{m}})^{k}A^{k}X=(A^{{\textcircled{w}}_{m}})^{k-1}(A^{{\textcircled{w}}_{m}}A)A^{k-1}X=(A^{{\textcircled{w}}_{m}})^{k-1}A^{k-1}X=\cdots=A^{{\textcircled{w}}_{m}}AX=X$$
and
\begin{eqnarray*}
(A^{{\textcircled{w}}_{m}})^{k}A_{1}^{k}A^{\dag}&=&(A^{{\textcircled{w}}_{m}})^{k-1}(A^{{\textcircled{w}}_{m}}A)A^{k-1}A^{{\textcircled{w}}_{m}}AA^{\dag} =(A^{{\textcircled{w}}_{m}})^{k-1}A^{k-1}A^{{\textcircled{w}}_{m}}AA^{\dag}\\&=&\cdots=A^{{\textcircled{w}}_{m}}AA^{{\textcircled{w}}_{m}}AA^{\dag}=A^{\textcircled{w}_{m},\dag},
\end{eqnarray*}
which implies $X=A^{{\textcircled{w}}_{m},\dag}$. Then the rest of proof is trivial.

$(e)\Rightarrow(f)$. Postmultiplying $XA=A^{{\textcircled{w}}_{m}}A$ by $A^{*}$, we get that $$XAA^{*}=A^{{\textcircled{w}}_{m}}AA^{*}.$$

  $(f)\Rightarrow(a)$. Postmultiplying $XAA^{*}=A^{{\textcircled{w}}_{m}}AA^{*}$ by $(A^{\dag})^{*}$$A^{{\textcircled{w}}_{m},\dag}$, we get that $$XAA^{{\textcircled{w}}_{m},\dag}=A^{{\textcircled{w}}_{m}}AA^{{\textcircled{w}}_{m},\dag}.$$ By $\mathcal{N}(X)=\mathcal{N}((A^{k})^{*}A^{m+1}A^{\dag})$, Lemma $\ref{l2.4}$ $(c)$ and Theorem $\ref{th3.10}$ $(c)$ and $(d)$, it follows that
$$XAA^{{\textcircled{w}}_{m},\dag}=X \ \ {\rm and} \ \ A^{{\textcircled{w}}_{m}}AA^{{\textcircled{w}}_{m},\dag}=A^{{\textcircled{w}}_{m},\dag},$$
which completes the proof.
\end{proof}
\end{theorem}

Using $X=XAX$ which is one of the classical Penrose equations \cite{WWQ}, we characterize the $m$-WGMP inverse.
\begin{theorem}
Let $A\in\mathbb{C}^{n\times n}_{k}$ and $m\in\mathbb Z^{+}$. Let $A_{1}$ be as in Theorem $\ref{def3.1}$. Then the following conditions are equivalent:
\begin{itemize}
\item[(a)] $X=A^{{\textcircled{w}}_{m},\dag}$;
\item[(b)] $XAX=X$, $AX=A_{1}A^{\dag}$ and $XAA^{*}=A^{\textcircled{w}_{m}}AA^{*}$;
\item[(c)] $XAX=X$, $AX=A_{1}A^{\dag}$ and $XA^{k}=A^{\textcircled{w}_{m}}A^{k}$;
\item[(d)] $XAX=X$, $XA=A^{\textcircled{w}_{m}}A$ and $A^{k}X=A_{1}^{k}A^{\dag}$;
\item[(e)] $XAX=X$, $XA=A^{\textcircled{w}_{m}}A$ and $A^{*}AX=A^{*}A_{1}A^{\dag}$.
\end{itemize}
\begin{proof}
That $(a)$ implies all other items $(b)-(e)$ can be directly verified by Theorem $\ref{th4.1}$ $(c)$, $(d)$ and $(f)$ and Definition $\ref{def3.5}$.

$(b)\Rightarrow(a)$. Since $XAX=X$ and $AX=A_{1}A^{\dag}$, we get that $\mathcal{N}(X)=\mathcal{N}(AX)=\mathcal{N}((A^{k})^{*}A^{m+1}A^{\dag}).$ By Theorem \ref{th4.1} $(f)$, we get $X=A^{{\textcircled{w}}_{m},\dag}.$

 $(c)\Rightarrow(a)$. By Theorem \ref{th3.10} $(c)$ and $(d)$, we get that $AA^{\textcircled{w}_{m}}=A(AA^{\textcircled{w}_{m}})A^{\textcircled{w}_{m}}=\cdots=A^{k}(A^{\textcircled{w}_{m}})^{k}$, together with the conditions of $(b)$, it follows that \begin{eqnarray*}
 X&=&XAX=XA_{1}A^{\dag}=XAA^{\textcircled{w}_{m}}AA^{\dag}=XA^{k}(A^{\textcircled{w}_{m}})^{k}AA^{\dag}\\&=&A^{\textcircled{w}_{m}}A^{k}(A^{\textcircled{w}_{m}})^{k}AA^{\dag}
 =A^{\textcircled{w}_{m}}AA^{\textcircled{w}_{m}}AA^{\dag}=A^{\textcircled{w}_{m},\dag}.
 \end{eqnarray*}

 $(d)\Rightarrow(a)$. By $XAX=X$ and $XA=A^{\textcircled{w}_{m}}A$, it follows that $\mathcal{R}(X)=\mathcal{R}(XA)=\mathcal{R}(A^{\textcircled{w}_{m}}A)=\mathcal{R}(A^{k})$. According to Theorem \ref{th4.1} $(d)$, we get $X=A^{{\textcircled{w}}_{m},\dag}.$

  $(e)\Rightarrow(a)$. Since $XAX=X$ and $XA=A^{\textcircled{w}_{m}}A$, it follows that $\mathcal{R}(X)=\mathcal{R}(XA)=\mathcal{R}(A^{\textcircled{w}_{m}}A)=\mathcal{R}(A^{k})$. By Theorem \ref{th4.1} $(c)$, we get $X=A^{{\textcircled{w}}_{m},\dag}.$
\end{proof}
\end{theorem}

Notice that the Drazin inverse is defined by using $A^{D}A^{k+1}=A^{k}$ and other two equations \cite{D}. By the fact that $A^{{\textcircled{w}}_{m},\dag}A^{k+1}=A^{k}$, we present similarly the characterizations of the $m$-WGMP inverse.
\begin{theorem}\label{th4.4}
Let $A\in\mathbb{C}^{n\times n}_{k}$ and $m\in\mathbb Z^{+}$. Let $A_{1}$ be as in Theorem $\ref{def3.1}$. Then the following statements are equivalent:
\begin{itemize}
\item[(a)] $X=A^{{\textcircled{w}}_{m},\dag}$;
\item[(b)]  $XA^{k+1}=A^{k}$, $r(X)=r(A^{k})$ and $A^{*}AX=A^{*}A_{1}A^{\dag}$;
\item[(c)] $XA^{k+1}=A^{k}$, $r(X)=r(A^{k})$ and $AX=A_{1}A^{\dag}$;
\item[(d)] $XA^{k+1}=A^{k}$, $r(X)=r(A^{k})$ and $A^{k}X=A_{1}^{k}A^{\dag}$.
\end{itemize}

\begin{proof}
 $(a)\Rightarrow(b)$. It is a direct consequence of Theorem \ref{th3.10} $(e)$ and Theorem \ref{th4.1} $(c)$.

  $(b)\Rightarrow(c)$. Premultiplying $A^{*}AX=A^{*}A_{1}A^{\dag}$ by $(A^{\dag})^{*}$, we get $AX=AA^{\dag}AA^{{\textcircled{w}}_{m}}AA^{\dag}=A_{1}A^{\dag}.$

  $(c)\Rightarrow(d)$. Premultiplying $AX=A_{1}A^{\dag}$ by $A^{k-1}$, and using $(\ref{eq...})$, we have $A^{k}X=A^{k}A^{{\textcircled{w}}_{m}}AA^{\dag}=A_{1}^{k}A^{\dag}$.

  $(d)\Rightarrow(a)$. From $XA^{k+1}=A^{k}$ and $r(X)=r(A^{k})$, we get $\mathcal{R}(X)=\mathcal{R}(A^{k})$. Hence $X=A^{{\textcircled{w}}_{m},\dag}$ by Theorem $\ref{th4.1}$ $(d)$.
\end{proof}
\end{theorem}

It is well-known that $X=A^{\dag}$ if and only if $AX=P_{A}$ and $XA=P_{A^{*}}$. However, we have that $AX=P_{\mathcal{R}(A^{k}),\mathcal{N}((A^{k})^{*}A^{m+1}A^{\dag})}$ and
$XA=P_{\mathcal{R}(A^{k}),\mathcal{N}((A^{k})^{*}A^{m+1})}$ when $X=A^{\textcircled{w}_{m},\dag}$. Conversely, it is invalid.
\begin{sloppypar}
\begin{example}
Let
$A= \left[\begin{array}{cc}
I_{3} & L\\
0 & N \\
\end{array}
\right]
 $and
 $X= \left[\begin{array}{cc}
I_{3} & 0\\
0 & L \\
\end{array}
\right],
 $
 where $N= \left[\begin{array}{ccc}
0 & 1 & 0\\
0 & 0 & 1\\
0 & 0 & 0\\
\end{array}
\right]
 $ and
 $L= \left[\begin{array}{ccc}
0 & 0 & 1\\
0 & 0 & 0\\
0 & 0 & 0\\
\end{array}
\right].
 $
 Then it is clear that $k=${\rm Ind}$(A)=3$ and
$A^{\textcircled{w}_{2},\dag}= \left[\begin{array}{cccccc}
I_{3} & 0\\
0 & 0 \\
\end{array}
\right].
$
It can be directly verified that $AX=P_{\mathcal{R}(A^{3}),\mathcal{N}((A^{3})^{*}A^{3}A^{\dag})}$ and $
XA=P_{\mathcal{R}(A^{3}),\mathcal{N}((A^{3})^{*}A^{3})}$. However, $X\neq A^{\textcircled{w}_{2},\dag}.$
\end{example}
\end{sloppypar}

 Therefore, we characterize the $m$-WGMP inverse by $AX=P_{\mathcal{R}(A^{k}),\mathcal{N}((A^{k})^{*}A^{m+1}A^{\dag})}$ and
$XA=P_{\mathcal{R}(A^{k}),\mathcal{N}((A^{k})^{*}A^{m+1})}$.
\begin{theorem}\label{th4.44}
Let $A\in\mathbb{C}^{n\times n}_{k}$ and $m\in\mathbb Z^{+}$. Then the following statements are equivalent:
\begin{itemize}
\item[(a)] $X=A^{{\textcircled{w}}_{m},\dag}$;
\item[(b)] $AX=P_{\mathcal{R}(A^{k}),\mathcal{N}((A^{k})^{*}A^{m+1}A^{\dag})}, XA=P_{\mathcal{R}(A^{k}),\mathcal{N}((A^{k})^{*}A^{m+1})}$ and $r(X)=r(A^{k})$;
\item[(c)] $AX=P_{\mathcal{R}(A^{k}),\mathcal{N}((A^{k})^{*}A^{m+1}A^{\dag})}, XA=P_{\mathcal{R}(A^{k}),\mathcal{N}((A^{k})^{*}A^{m+1})}$ and $AX^{2}=X$;
\item[(d)] $AX=P_{\mathcal{R}(A^{k}),\mathcal{N}((A^{k})^{*}A^{m+1}A^{\dag})}, XA=P_{\mathcal{R}(A^{k}),\mathcal{N}((A^{k})^{*}A^{m+1})}$ and $XAX=X$.
\end{itemize}

\begin{proof}
 $(a)\Rightarrow(b)$. It is a direct consequence of Theorem \ref{th3.10} $(a)$, $(d)$ and $(e)$.

  $(b)\Rightarrow(c)$. Since $AX=P_{\mathcal{R}(A^{k}),\mathcal{N}((A^{k})^{*}A^{m+1}A^{\dag})}$ and $r(X)=r(A^{k})$, we get $AX^{2}=X$.

  $(c)\Rightarrow(d)$. From $AX^{2}=X$, we have $\mathcal{R}(X)=\mathcal{R}(AX^{2})=\cdots=\mathcal{R}(A^{k}X^{k+1})\subseteq \mathcal{R}(A^{k})$.
  Since $XA=P_{\mathcal{R}(A^{k}),\mathcal{N}((A^{k})^{*}A^{m+1})}$, $XAX=X$ immediately follows.

  $(d)\Rightarrow(a)$. By the conditions of $(d)$, we get that $\mathcal{R}(X)=\mathcal{R}(XA)=\mathcal{R}(A^{k})$ and $\mathcal{N}(X)=\mathcal{N}(AX)=\mathcal{N}((A^{k})^{*}A^{m+1}A^{\dag})$.
  Hence $X=A^{{\textcircled{w}}_{m},\dag}$ by Theorem \ref{th3.10} $(c)$.
\end{proof}
\end{theorem}

\section{Representations of the $m$-WGMP inverse}
In \cite{WLJ}, the authors provide some expressions of the $m$-weak group inverse by some generalized inverses. Based on them, we directly get several representations of the $m$-WGMP inverse.
\begin{theorem}\label{t5.1}
Let $A\in\mathbb{C}^{n\times n}_{k}$ and $m\in\mathbb{Z^{+}}$. Then the following statements hold:
\begin{itemize}
	\item[$(a)$]  $A^{\textcircled{w}_{m},\dag}=(A^{D})^{m+1}P_{A^{k}}A^{m}P_{A};$

	\item[$(b)$] $A^{\textcircled{w}_{m},\dag}=A^{k-m}(A^{k+1})^{\textcircled{\#}}A^{m}P_{A} \ (k\geq m);$

	\item[$(c)$] $A^{\textcircled{w}_{m},\dag}=(A^{k})^{\#}A^{k-m-1}P_{A^{k}}A^{m}P_{A} \ (k\geq m+1);$

	\item[$(d)$] $A^{\textcircled{w}_{m},\dag}=(A^{m+1}P_{A^{k}})^{\dag}A^{m}P_{A};$

	\item[$(e)$] $A^{\textcircled{w}_{m},\dag}=A^{m-1}P_{A^{k}}(A^{m})^{\textcircled{w}}P_{A}.$
\end{itemize}
\begin{proof}
It is a direct consequence of $(\ref{th3,7})$ and \cite[Theorem 5.1]{WLJ}.
\end{proof}
\end{theorem}

 To present some limit expressions of the $m$-WGMP inverse, we need the following lemma.
\begin{lemma}{\rm\cite{YZ, GGA}}\label{l5.4}
Let $A\in\mathbb{C}^{m\times n}, X\in\mathbb{C}^{n\times p}$ and $Y\in\mathbb{C}^{p\times m}.$ Then the following statements are equivalent:
\begin{itemize}
	\item[$(a)$]  $\lim_{\lambda \to 0}X(\lambda I_{p}+YAX)^{-1}Y$ exists;

	\item[$(b)$] $A^{(2)}_{\mathcal{R}(XY),\mathcal{N}(XY)} $ exists,
\end{itemize}in which case
$$\begin{matrix}\lim_{\lambda \to 0}X(\lambda I_{p}+YAX)^{-1}Y=A^{(2)}_{\mathcal{R}(XY),\mathcal{N}(XY)}\end{matrix}.$$
Furthermore
\begin{equation}\label{xx5.1}
A^{\dag}=\begin{matrix}\lim_{\lambda \to 0}{(\lambda I_{n}+A^{*}A)^{-1}A^{*}}=\lim_{\lambda \to 0}{A^{*}(\lambda I_{n}+AA^{*})^{-1}}\end{matrix}.
\end{equation}

\end{lemma}
\begin{theorem}\label{thm3.15t}
Let $A\in\mathbb{C}^{n\times n}_{k}$ and $m\in\mathbb Z^{+}$. Then the following assertions are hold:
\begin{itemize}
\item[(a)] $A^{{\textcircled{w}}_{m}, \dag}=\begin{matrix}\lim_{\lambda \to 0}{A^{k}(\lambda I_{n}+A^{k}(A^{k})^{*}A^{m+1})^{-1}}(A^{k})^{*}A^{m+1}(\lambda I_{n}+A^{*}A)^{-1}A^{*}
\end{matrix};$
\item[(b)] $A^{{\textcircled{w}}_{m}, \dag}=\begin{matrix}\lim_{\lambda \to 0}{A^{k}(A^{k})^{*}(\lambda I_{n}+A^{k+m+1}(A^{k})^{*})^{-1}}A^{m+1}(\lambda I_{n}+A^{*}A)^{-1}A^{*}
\end{matrix};$
\item[(c)] $A^{{\textcircled{w}}_{m}, \dag}=\begin{matrix}\lim_{\lambda \to 0}{(\lambda I_{n}+A^{k}(A^{k})^{*}A^{m+1})^{-1}}A^{k}(A^{k})^{*}A^{m+1}(\lambda I_{n}+A^{*}A)^{-1}A^{*}
\end{matrix}.$
\end{itemize}
\begin{proof}
$(a)$. It is easy to check that $r(A^{k}(A^{k})^{*}A^{m+1}A^{\dag})=r((A^{k})^{*}A^{m+1}A^{\dag})=r(A^{k})=t$, which implies that
 $\mathcal{R}(A^{k})=\mathcal{R}(A^{k}(A^{k})^{*}A^{m+1}A^{\dag})$ and $\mathcal{N}((A^{k})^{*}A^{m+1}A^{\dag})=\mathcal{N}(A^{k}(A^{k})^{*}A^{m+1}A^{\dag})$.
From Theorem $\ref{th3.10}$ $(c)$, we get
$$
A^{\textcircled{w}_{m}}=A^{(2)}_{\mathcal{R}(A^{k}),\mathcal{N}((A^{k})^{*}A^{m+1}A^{\dag})}=
A^{(2)}_{\mathcal{R}(A^{k}(A^{k})^{*}A^{m+1}A^{\dag}),\mathcal{N}(A^{k}(A^{k})^{*}A^{m+1}A^{\dag})}.$$
 Let $X=A^{k}$ and $Y=(A^{k})^{*}A^{m+1}A^{\dag}$. By  Lemma \ref{l5.4} and $(\ref{xx5.1})$, we get that
$$A^{\textcircled{w}_{m}}=\begin{matrix}\lim_{\lambda \to 0}{A^{k}(\lambda I_{n}+A^{k}(A^{k})^{*}A^{m+1})^{-1}}(A^{k})^{*}A^{m+1}(\lambda I_{n}+A^{*}A)^{-1}A^{*}
\end{matrix}.$$

\noindent The statements $(b)$--$(c)$ can be similarly proved.
\end{proof}
\end{theorem}

\begin{remark}\label{rem5.5}
$(\lambda I_{n}+A^{*}A)^{-1}A^{*}$ in Theorem \ref{thm3.15t} $(a)$, $(b)$ and $(c)$ can be replaced by $A^{*}(\lambda I_{n}+AA^{*})^{-1}.$

\end{remark}

The following example will illustrate the effectiveness of the expressions in Theorem \ref{thm3.15t} for computing the $m$-WGMP inverse.
\begin{example}\label{ex5.6}
Let
\begin{eqnarray*}
\tiny A=\left[
\begin{array}{ccccccc}
   4.8990 + 7.3786i&   6.8197 + 3.0145i&   7.2244 + 1.2801i&   4.5380 + 1.9043i&   8.3138 + 3.7627i&   6.2797 + 3.8462i&   3.7241 + 9.8266i\\
   1.6793 + 2.6912i &  0.4243 + 7.0110i &  1.4987 + 9.9908i&   4.3239 + 3.6892i&   8.0336 + 1.9092i&   2.9198 + 5.8299i&   1.9812 + 7.3025i\\
   9.7868 + 4.2284i&   0.7145 + 6.6634i &  6.5961 + 1.7112i&   8.2531 + 4.6073i&   0.6047 + 4.2825i&   4.3165 + 2.5181i&   4.8969 + 3.4388i\\
   7.1269 + 5.4787i &  5.2165 + 5.3913i&   5.1859 + 0.3260i&   0.8347 + 9.8164i&   3.9926 + 4.8202i&   0.1549 + 2.9044i&   3.3949 + 5.8407i\\
   0.0000 + 0.0000i&   0.0000 + 0.0000i&   0.0000 + 0.0000i&   0.0000 + 0.0000i&   0.0000 + 0.0000i&   9.8406 + 6.1709i&   0.0000 + 0.0000i\\
   0.0000 + 0.0000i&   0.0000 + 0.0000i&   0.0000 + 0.0000i&   0.0000 + 0.0000i&   0.0000 + 0.0000i&   0.0000 + 0.0000i&   9.2033 + 9.0631i\\
   0.0000 + 0.0000i&   0.0000 + 0.0000i&   0.0000 + 0.0000i&   0.0000 + 0.0000i&   0.0000 + 0.0000i&   0.0000 + 0.0000i&   0.0000 + 0.0000i
   \end{array}
   \right]
\end{eqnarray*}
 It can be checked that $k=${\rm Ind}$(A)=3$. By the function pinv which is used to calculate the MP-inverse of a given matrix, we get
 \begin{eqnarray*}
 A^{\dag}= \tiny\left[
\begin{array}{ccccccc}
  -0.0228 - 0.0879i&   0.0168 + 0.0323i&   0.0512 - 0.0142i&   0.0139 + 0.0485i&   0.0025 + 0.0322i&  -0.0346 + 0.0354i&   0.0000 + 0.0000i\\
   0.0109 + 0.0512i&  -0.0172 - 0.0333i&  -0.0518 - 0.0195i&   0.0500 - 0.0302i&   0.0005 - 0.0096i&   0.0040 - 0.0063i&   0.0000 + 0.0000i\\
   0.0039 + 0.0192i&  -0.0188 - 0.0571i&   0.0155 + 0.0020i&   0.0026 + 0.0171i&  -0.0139 + 0.0171i&  -0.0100 + 0.0102i&   0.0000 + 0.0000i\\
   0.0263 + 0.0533i&  -0.0129 - 0.0019i&   0.0427 + 0.0287i&  -0.0611 - 0.0841i&  -0.0385 - 0.0161i&   0.0022 - 0.0026i&   0.0000 + 0.0000i\\
   0.0421 - 0.0455i&   0.0502 + 0.0334i&  -0.0428 - 0.0124i&   0.0016 + 0.0096i&  -0.0204 + 0.0015i&  -0.0411 - 0.0151i&   0.0000 + 0.0000i\\
  -0.0000 - 0.0000i&   0.0000 + 0.0000i&   0.0000 - 0.0000i&   0.0000 + 0.0000i&   0.0729 - 0.0457i&   0.0000 + 0.0000i&   0.0000 + 0.0000i\\
  -0.0000 + 0.0000i&   0.0000 + 0.0000i&   0.0000 - 0.0000i&   0.0000 + 0.0000i&   0.0000 - 0.0000i&   0.0552 - 0.0543i&   0.0000 + 0.0000i
   \end{array}
   \right].
\end{eqnarray*}
 By $(\ref{eq.1})$, we get

$ A^{\textcircled{\emph{\dag}}}=(A^{4}(A^{3})^{\dag})^{\dag}=$
 \begin{eqnarray*}\tiny\left[
\begin{array}{ccccccc}
  -0.0032 - 0.1411i&   0.0709 + 0.0429i&   0.0106 - 0.0097i&   0.0187 + 0.0560i &  0.0000 + 0.0000i&   0.0000 + 0.0000i &  0.0000 + 0.0000i\\
   0.0878 + 0.0001i&   0.0438 + 0.0326i&  -0.1096 - 0.0524i&   0.0490 - 0.0158i&   0.0000 + 0.0000i &  0.0000 + 0.0000i &  0.0000 + 0.0000i\\
  -0.0896 + 0.0173i & -0.0314 - 0.1473i &  0.0445 + 0.0629i&   0.0118 + 0.0056i&   0.0000 + 0.0000i &  0.0000 + 0.0000i&   0.0000 + 0.0000i\\
   0.0431 + 0.0788i & -0.0348 + 0.0182i&   0.0536 + 0.0096i & -0.0658 - 0.0844i&   0.0000 + 0.0000i &  0.0000 + 0.0000i &  0.0000 + 0.0000i\\
   0.0000 + 0.0000i &  0.0000 + 0.0000i &  0.0000 + 0.0000i&   0.0000 + 0.0000i&   0.0000 + 0.0000i &  0.0000 + 0.0000i &  0.0000 + 0.0000i\\
   0.0000 + 0.0000i&   0.0000 + 0.0000i &  0.0000 + 0.0000i&   0.0000 + 0.0000i&   0.0000 + 0.0000i &  0.0000 + 0.0000i &  0.0000 + 0.0000i\\
   0.0000 + 0.0000i &  0.0000 + 0.0000i&   0.0000 + 0.0000i&   0.0000 + 0.0000i&   0.0000 + 0.0000i&   0.0000 + 0.0000i &  0.0000 + 0.0000i
   \end{array}
   \right].
\end{eqnarray*}
Assume $m=2$. By $(\ref{eq.4})$, it follows that

$ A^{\textcircled{w}_{2},\dag}=
 A^{\textcircled{w}_{2}}AA^{\dag}=(A^{\textcircled{\emph{\dag}}})^3A^{3}A^{\dag}=$ \begin{eqnarray*}\tiny\left[
\begin{array}{ccccccc}
  -0.0032 - 0.1411i&   0.0709 + 0.0429i&   0.0106 - 0.0097i&   0.0187 + 0.0560i&  -0.0151 - 0.0310i&  -0.2502 + 0.2087i&  -0.1223 + 0.0627i\\
   0.0878 + 0.0001i&   0.0438 + 0.0326i&  -0.1096 - 0.0524i&   0.0490 - 0.0158i&  0.1788 + 0.2332i&   -0.3818 + 0.4296i&  -0.1105 + 0.0067i\\
  -0.0896 + 0.0173i&  -0.0314 - 0.1473i&   0.0445 + 0.0629i&   0.0118 + 0.0056i&  -0.0415 - 0.2864i&   0.7974 - 0.1768i&   0.2687 + 0.0926i\\
   0.0431 + 0.0788i&  -0.0348 + 0.0182i&   0.0536 + 0.0096i&  -0.0658 - 0.0844i&   0.0125 - 0.0262i&   0.0972 - 0.3512i&   0.0454 - 0.1317i\\
   0.0000 + 0.0000i&   0.0000 + 0.0000i&   0.0000 + 0.0000i&   0.0000 + 0.0000i&   0.0000 + 0.0000i&   0.0000 + 0.0000i&   0.0000 + 0.0000i\\
   0.0000 + 0.0000i&   0.0000 + 0.0000i&   0.0000 + 0.0000i&   0.0000 + 0.0000i&   0.0000 + 0.0000i&   0.0000 + 0.0000i&   0.0000 + 0.0000i\\
   0.0000 + 0.0000i&   0.0000 + 0.0000i&   0.0000 + 0.0000i&   0.0000 + 0.0000i&   0.0000 + 0.0000i&   0.0000 + 0.0000i&   0.0000 + 0.0000i
   \end{array}
   \right].
\end{eqnarray*}
Let $L_{1}=\begin{matrix}\lim_{\lambda \to 0}{A^{k}(\lambda I_{n}+A^{k}(A^{k})^{*}A^{m+1})^{-1}}(A^{k})^{*}A^{m+1}(\lambda I_{n}+A^{*}A)^{-1}A^{*}
\end{matrix}$. Then $L_{1}=$
\begin{eqnarray*}
 \small\tiny\left[
\begin{array}{ccccccc}
- 0.003189 - 0.1411i&   0.07088 + 0.04294i& 0.01062 - 0.009663i&   0.01868 + 0.05603i& - 0.01512 - 0.03101i& - 0.2502 + 0.2087i& 0.0000 + 0.0000i\\
  0.08776 + 0.00006i&   0.04379 + 0.03259i& - 0.1096 - 0.05243i&   0.04899 - 0.01576i&     0.1788 + 0.2332i& - 0.3818 + 0.4296i& 0.0000 + 0.0000i\\
 - 0.08964 + 0.01729i&  - 0.03136 - 0.1473i&  0.04449 + 0.06285i&  0.01184 + 0.005646i&  - 0.04154 - 0.2864i&   0.7974 - 0.1768i& 0.0000 + 0.0000i\\
   0.04314 + 0.07879i& - 0.03482 + 0.01822i& 0.05364 + 0.009637i& - 0.06583 - 0.08442i&   0.01248 - 0.02618i&  0.09717 - 0.3512i& 0.0000 + 0.0000i\\
   0.0000 + 0.0000i&   0.0000 + 0.0000i&   0.0000 + 0.0000i&   0.0000 + 0.0000i&   0.0000 + 0.0000i&   0.0000 + 0.0000i&   0.0000 + 0.0000i\\
   0.0000 + 0.0000i&   0.0000 + 0.0000i&   0.0000 + 0.0000i&   0.0000 + 0.0000i&   0.0000 + 0.0000i&   0.0000 + 0.0000i&   0.0000 + 0.0000i\\
   0.0000 + 0.0000i&   0.0000 + 0.0000i&   0.0000 + 0.0000i&   0.0000 + 0.0000i&   0.0000 + 0.0000i&   0.0000 + 0.0000i&   0.0000 + 0.0000i
   \end{array}
   \right]
\end{eqnarray*}
and $$r_{1}=\parallel A^{\textcircled{w}_{2},\dag}-L_{1} \parallel= 5.753\times10^{-11},$$ where $\parallel \cdot \parallel$ is the Frobenius norm.
Let
$$ L_{2}=\lim_{\lambda \to 0}{A^{3}(A^{3})^{*}(\lambda I_{n}+(A^{3})^{*}A^{6})^{-1}}A^{3}(\lambda I_{n}+A^{*}A)^{-1}A^{*}$$and
$$L_{3}=\lim_{\lambda \to 0}{(\lambda I_{n}+(A^{3})^{*}A^{6})^{-1}}A^{3}(A^{3})^{*}A^{3}(\lambda I_{n}+A^{*}A)^{-1}A^{*}. $$
Similarly, we get that
$$r_{2}=\parallel A^{\textcircled{w}_{2},\dag}-L_{2} \parallel= 7.595\times10^{-15}\ \ and\ \ r_{3}=\parallel A^{\textcircled{w}_{2},\dag}-L_{3} \parallel= 1.214\times10^{-16}.$$
\end{example}

\section{Applications of the $m$-WGMP inverse}
 Using the $m$-WGMP inverse, we solve the restricted matrix equation below.

 \begin{theorem} Let $A\in\mathbb{C}^{n\times n}_{k}$, $X\in\mathbb{C}^{q\times n}$ and $D\in\mathbb{C}^{q\times n}$. If $\mathcal{N}(D)\supseteq\mathcal{N}((A^{k})^{*}A^{m+1})$, then the restricted matrix equation
\begin{equation}\label{e7.1}
XA=D,\ \ \ \mathcal{N}(X)\supseteq\mathcal{N}((A^{k})^{*}A^{m+1}A^{\dag})
\end{equation}
has the unique solution $X=DA^{\textcircled{w}_{m},\dag}$.
\end{theorem}
\begin{proof}
Notice that $(\ref{e7.1})$ is equivalent to the restricted
matrix equation
\begin{eqnarray*}
A^{*}X^{*}=D^{*},\ \ \ \mathcal{R}(X^{*})\subseteq\mathcal{R}((A^{m+1}A^{\dag})^{*}A^{k}).
\end{eqnarray*}
By $\mathcal{N}(D)\supseteq\mathcal{N}((A^{k})^{*}A^{m+1})$, we have $\mathcal{R}(D^{*})\subseteq\mathcal{R}((A^{m+1})^{*}A^{k})$. Since $A^{*}\mathcal{R}((A^{m+1}A^{\dag})^{*}A^{k})=\mathcal{R}((A^{m+1})^{*}A^{k})$ and $\mathcal{R}(D^{*})\subseteq\mathcal{R}((A^{m+1})^{*}A^{k})$, we get that $\mathcal{R}(D^{*})\subseteq A^{*}\mathcal{R}((A^{m+1}A^{\dag})^{*}A^{k})$, which implies solvability of $(\ref{e7.1})$. Obviously, $X=DA^{\textcircled{w}_{m},\dag}$ is a solution of $(\ref{e7.1})$.
Then we prove the uniqueness of $X$. If $X_{1}$ also satisfies $(\ref{e7.1})$, then $$X=DA^{\textcircled{w}_{m},\dag}=X_{1}AA^{\textcircled{w}_{m},\dag}=X_{1}P_{\mathcal{R}(A^{k}), \mathcal{N}((A^{k})^{*}A^{m+1}A^{\dag})}=X_{1},$$
which finishes the proof.
\end{proof}
The next theorem shows the inverse of the nonsingular bordered matrix constructed by the $m$-WGMP inverse.
\begin{theorem}\label{t7.1}
Let $A\in\mathbb{C}^{n\times n}_{k}$ be such that  $r(A^{k})=t$ and $m\in\mathbb{Z^{+}}$. Let $B\in\mathbb{C}^{n\times (n-t)}$ and $C^{*}\in\mathbb{C}^{n\times (n-t)}$ be of full column rank such that
$\mathcal{R}(B)=\mathcal{N}((A^{k})^{*}A^{m+1}A^{\dag})$ and $\mathcal{N}(C)=\mathcal{R}(A^{k})$.
Then the bordered matrix
\begin{equation}\label{al2}
K=\left[\begin{array}{cc}
A & B\\
C & 0\\
\end{array}\right]
\end{equation}
is nonsingular and its inverse is given by
\begin{eqnarray*}
K^{-1}=\left[\begin{array}{cc}
A^{\textcircled{w}_{m},\dag} &\ \ (I_{n}-A^{\textcircled{w}_{m},\dag}A)C^{\dag}\\
B^{\dag}(I_{n}-AA^{\textcircled{w}_{m},\dag}) &\ \ B^{\dag}(AA^{\textcircled{w}_{m},\dag}A-A)C^{\dag}\\
\end{array}\right].
\end{eqnarray*}\end{theorem}
\begin{proof}Let
$
M=\left[\begin{array}{cc}
A^{\textcircled{w}_{m},\dag} &\ \ (I_{n}-A^{\textcircled{w}_{m},\dag}A)C^{\dag}\\
B^{\dag}(I_{n}-AA^{\textcircled{w}_{m},\dag}) &\ \ B^{\dag}(AA^{\textcircled{w}_{m},\dag}A-A)C^{\dag}\\
\end{array}\right].
$
Since $\mathcal{R}(A^{\textcircled{w}_{m},\dag})=\mathcal{R}(A^{k})=\mathcal{N}(C)$, we have that $CA^{\textcircled{w}_{m},\dag}=0$. Since $C$ is a full row rank matrix, it is right invertible and  $CC^{\dag}=I_{n-t}$. From  $$\mathcal{R}(I_{n}-AA^{\textcircled{w}_{m},\dag})=\mathcal{N}(AA^{\textcircled{w}_{m},\dag})=\mathcal{N}((A^{k})^{*}A^{m+1}A^{\dag})=\mathcal{R}(B)=\mathcal{R}(BB^{\dag}),$$ we get
$BB^{\dag}(I_{n}-AA^{\textcircled{w}_{m},\dag})=I_{n}-AA^{\textcircled{w}_{m},\dag}.$
Hence
\begin{eqnarray*}
KM&=&\left[\begin{array}{cc}
AA^{\textcircled{w}_{m},\dag}+BB^{\dag}(I_{n}-AA^{\textcircled{w}_{m},\dag}) &\ \ A(I_{n}-A^{\textcircled{w}_{m},\dag}A)C^{\dag}+BB^{\dag}(AA^{\textcircled{w}_{m},\dag}A-A)C^{\dag}\\
CA^{\textcircled{w}_{m},\dag} &\ \ C(I_{n}-A^{\textcircled{w}_{m},\dag}A)C^{\dag}\\
\end{array}\right]\\
&=&\left[\begin{array}{cc}
AA^{\textcircled{w}_{m},\dag}+I_{n}-AA^{\textcircled{w}_{m},\dag} &\ \ A(I_{n}-A^{\textcircled{w}_{m},\dag}A)C^{\dag}-(I_{n}-AA^{\textcircled{w}_{m},\dag})AC^{\dag}\\
0 &\ \ CC^{\dag}\\
\end{array}\right]\\
&=&\left[\begin{array}{cc}
I_{n} &\ \ 0\\
0 &\ \ I_{n-t}\\
\end{array}\right].
\end{eqnarray*}
Hence $M=K^{-1}$. \end{proof}
Finally, based on $(\ref{al2})$, we get the Cramer's rule \cite{WWQ} for the solution of $(\ref{e7.1})$.
\begin{theorem}\label{t7.3}
Let $A\in\mathbb{C}^{n\times n}_{k}$, $B\in\mathbb{C}^{n\times (n-t)}$ and $C^{*}\in\mathbb{C}^{n\times (n-t)}$ be as in Theorem \ref{t7.1}. Let $X\in\mathbb{C}^{q\times n}$ and $D\in\mathbb{C}^{q\times n}$ be as in Theorem \ref{e7.1}. Then the unique solution  of $(\ref{e7.1})$ is  given by $X=[x_{ij}]$, where
\begin{equation}\label{tt21.2}
x_{ij}=\dfrac{{\rm det}\left[\begin{array}{cc}
A(d_{i}\leftarrow j) & B(0\leftarrow j)\\
C &0\\
\end{array}\right]}{{\rm det}\left[\begin{array}{cc}
A & B\\
C & 0\\
\end{array}\right]}, \ i=1 ,2, ..., q, j=1, 2,  ..., n,
\end{equation}where $d_{i}$ denotes the $i$-th row of $D$ and $A(d_{i}\leftarrow j)$ and $B(0\leftarrow j)$ mean to substitute the $j$-th row of $A$ and $B$ by $d_{i}$ and $0$, respectively.
\end{theorem}
\begin{proof}
Since $X$ is the solution of $(\ref{e7.1})$, we get that $\mathcal{N}(X)\supseteq\mathcal{N}((A^{k})^{*}A^{m+1}A^{\dag})=\mathcal{R}(B)$, which implies $XB=0$. Then $(\ref{e7.1})$ can be rewritten as
$$
\left[\begin{array}{cccc}
X & 0
\end{array}\right]\left[\begin{array}{cc}
A & B\\
C & 0\\
\end{array}\right]=\left[\begin{array}{cccc}
XA & XB
\end{array}\right]=\left[\begin{array}{cccc}
D&0
\end{array}\right].
$$
By Theorem \ref{t7.1}, we have that $\left[\begin{array}{cc}
A & B\\
C & 0\\
\end{array}\right]$ is nonsingular. Consequently, $X$ can be expressed by $(\ref{tt21.2})$ by using the classical Cramer's rule.
\end{proof}

\begin{example}
Let
\begin{eqnarray*}
\tiny A=\left[
\begin{array}{ccccccc}
   4.8990 + 7.3786i&   6.8197 + 3.0145i&   7.2244 + 1.2801i&   4.5380 + 1.9043i&   8.3138 + 3.7627i&   6.2797 + 3.8462i&   3.7241 + 9.8266i\\
   1.6793 + 2.6912i &  0.4243 + 7.0110i &  1.4987 + 9.9908i&   4.3239 + 3.6892i&   8.0336 + 1.9092i&   2.9198 + 5.8299i&   1.9812 + 7.3025i\\
   9.7868 + 4.2284i&   0.7145 + 6.6634i &  6.5961 + 1.7112i&   8.2531 + 4.6073i&   0.6047 + 4.2825i&   4.3165 + 2.5181i&   4.8969 + 3.4388i\\
   7.1269 + 5.4787i &  5.2165 + 5.3913i&   5.1859 + 0.3260i&   0.8347 + 9.8164i&   3.9926 + 4.8202i&   0.1549 + 2.9044i&   3.3949 + 5.8407i\\
   0.0000 + 0.0000i&   0.0000 + 0.0000i&   0.0000 + 0.0000i&   0.0000 + 0.0000i&   0.0000 + 0.0000i&   9.8406 + 6.1709i&   0.0000 + 0.0000i\\
   0.0000 + 0.0000i&   0.0000 + 0.0000i&   0.0000 + 0.0000i&   0.0000 + 0.0000i&   0.0000 + 0.0000i&   0.0000 + 0.0000i&   9.2033 + 9.0631i\\
   0.0000 + 0.0000i&   0.0000 + 0.0000i&   0.0000 + 0.0000i&   0.0000 + 0.0000i&   0.0000 + 0.0000i&   0.0000 + 0.0000i&   0.0000 + 0.0000i
   \end{array}
   \right],
\end{eqnarray*}
 \begin{eqnarray*}\tiny B=\left[
\begin{array}{ccccccc}
  -0.1981 - 0.0535i&  -0.0151 + 0.0009i   &    0.0000 + 0.0000i\\
  -0.2455 - 0.0143i&   -0.0292 + 0.0316i   &    0.0000 - 0.0000i\\
  -0.0740 - 0.0930i&   -0.1353 + 0.1981i   &   -0.0000 - 0.0000i\\
  -0.0597 - 0.0996i&   -0.0039 + 0.1198i  & -0.0000 + 0.0000i\\
   0.6294 - 0.0000i&   -0.0743 - 0.3123i  & -0.0000 - 0.0000i\\
  -0.0743 + 0.3123i&    0.2299 - 0.0000i   & 0.0000 + 0.0000i\\
   0.0000 - 0.0000i&    0.0000 + 0.0000i   & 1.0000 + 0.0000i\\
   \end{array}
   \right],\ \
   \end{eqnarray*}
   \begin{eqnarray*}
  \tiny  C=\left[
\begin{array}{ccccccc}
   0.0000 + 0.0000i&   0.0000 + 0.0000i&   0.0000 + 0.0000i &  0.0000 + 0.0000i&   0.6047 + 4.2825i&   1.5961 + 1.7112i&   4.3165 + 4.3181i\\
   0.0000 + 0.0000i&   0.0000 + 0.0000i&   0.0000 + 0.0000i&   0.0000 + 0.0000i&   3.7241 + 9.8266i&   6.3138 + 4.7627i&   0.0000 + 6.3634i\\
   0.0000 + 0.0000i&   0.0000 + 0.0000i&   0.0000 + 0.0000i&   0.0000 + 0.0000i&   0.0000 + 0.0000i&   0.0000 + 0.0000i&   0.0000 + 9.2908i
   \end{array}
   \right]
\end{eqnarray*}
and
\begin{eqnarray*}
\tiny D&=&
\footnotesize{\tiny\left[
	\begin{array}{cccccccc}
 1.5062\times10^{9} - 4.8185\times10^{7}i& 1.2087\times10^{9} + 3.5089\times10^{8}i& 1.0936\times10^{9} - 1.6609\times10^{8}i\\
  1.3536\times10^{9} + 8.04522\times10^{5}i&  1.076\times10^{9} + 3.5073\times10^{8}i& 9.8797\times10^{8} - 1.1718\times10^{8}i\\
 1.1886\times10^{9} - 1.8899\times10^{7}i& 9.4962\times10^{8} + 2.9231\times10^{8}i& 8.6552\times10^{8} - 1.1718\times10^{8}i\\
 \end{array}
	\right.
}\end{eqnarray*}$$
 \footnotesize{ \left.
	\tiny	\begin{array}{ccccccc}
   1.2534\times10^{9} + 1.5442\times10^{8}i&  1.3783\times10^{9} - 1.264\times10^{8}i& 1.5788\times10^{9} - 3.6923\times10^{7}i& 2.1042\times10^{9} + 5.7947\times10^{8}i\\
  1.1206\times10^{9} + 1.7408\times10^{8}i& 1.2418\times10^{9} - 7.3545\times10^{7}i& 1.4197\times10^{9} + 1.2161\times10^{7}i& 1.8729\times10^{9} + 5.8204\times10^{8}i\\
   9.866\times10^{8} + 1.3695\times10^{8}i& 1.0892\times10^{9} - 8.2455\times10^{7}i& 1.2465\times10^{9} - 9.6168\times10^{6}i& 1.6528\times10^{9} + 4.8392\times10^{8}i\\
	\end{array}
	\right] }.$$
It is clear that $k=${\rm Ind}$(A)=3$. Assume $m=2$, it can be verified that $\mathcal{R}(B)=\mathcal{N}((A^{3})^{*}A^{2+1}A^{\dag})$, $\mathcal{N}(C)=\mathcal{R}(A^{3})$ and $\mathcal{N}(D)\supseteq\mathcal{N}((A^{3})^{*}A^{2+1})$.
Let $X_{1}=DA^{\textcircled{w}_{2},\dag}$ and let $X_{2}$ be given by $(\ref{tt21.2})$. Then we get that

$ X_{1}=10^{7}\times$
\begin{eqnarray*}
   \tiny \left[
\begin{array}{ccccccc}
   4.1209 - 4.2316i&   4.5108 - 2.2349i&    2.6300 - 3.5180i&    3.9987 - 3.0137i&    3.6710 - 3.8593i&    3.9631 - 3.9288i&    0.0000 + 0.0000i\\
   3.8306 - 3.6777i&    4.1158 - 1.8845i&    2.4601 - 3.0899i&    3.6761 - 2.5791i&    3.4097 - 3.3626i&    3.6749 - 3.4133i&    0.0000 + 0.0000i\\
   3.3086 - 3.2851i&   3.5868 - 1.7117i&    2.1165 - 2.7467i&    3.1912 - 2.3206i&    2.9453 - 3.0006i&    3.1768 - 3.0497i&    0.0000 + 0.0000i
   \end{array}
   \right]
\end{eqnarray*}and

$ X_{2}=10^{7}\times$
\begin{eqnarray*}
   \tiny \left[
\begin{array}{ccccccc}
    4.1209 - 4.2316i& 4.5108 - 2.2349i& 2.6300 - 3.518i& 3.9987 - 3.0137i&  3.671 - 3.8593i& 3.9631 - 3.9288i& - 7.3049\times10^{-21} - 3.4222\times10^{-22}i\\
    3.8306 - 3.6777i& 4.1158 - 1.8845i& 2.4601 - 3.0899i& 3.6761 - 2.5791i& 3.4097 - 3.3626i& 3.6749 - 3.4133i&   5.6399\times10^{-21} + 4.7633\times10^{-22}i\\
    3.3086 - 3.2851i& 3.5868 - 1.7117i& 2.1165 - 2.7467i& 3.1912 - 2.3206i& 2.9453 - 3.0006i& 3.1768 - 3.0497i&   2.6649\times10^{-21} - 8.2521\times10^{-22}i\\
   \end{array}
   \right].
\end{eqnarray*}
Hence
$$r_{4}=\parallel X_{1}A-D \parallel= 1\times10^{-3}\ \ and\ \ r_{5}=\parallel X_{2}A-D \parallel= 3.0008\times10^{-4},$$
where $\parallel \cdot \parallel$ is the Frobenius norm, which implies the validity of Theorem \ref{e7.1} and Theorem \ref{t7.3} in solving $(\ref{e7.1})$.
\end{example}

\section{Conclusions}
This paper gives a notion of the $m$-core-nilpotent decomposition for the complex square matrices, which is a generalization of the core-nilpotent decomposition. By applying the $m$-core-nilpotent decomposition, Drazin inverse and MP-inverse, we introduce a new generalized inverse called the $m$-WGMP inverse (see Theorem \ref{thm3.4}). Several limit representations of the $m$-WGMP inverse are efficient in terms of the computation (see Example \ref{ex5.6}). The two ways of solving $(\ref{e7.1})$ are given by the $m$-WGMP inverse (see Theorem \ref{e7.1} and Theorem \ref{t7.3}).
Further investigations deserve more attention as follows:

$(a)$ Further properties, characterizations and representations of the $m$-WGMP inverse;

$(b)$ The relationships between the $m$-WGMP inverse and other generalized inverses;

$(c)$ The applications of the $m$-WGMP inverse, such as in partial order, in other restricted matrix equations, and in symbolic computations;

$(d)$ The perturbation formulaes, iterative algorithm and splitting method for computing the $m$-WGMP inverse.

\section*{Conflict of interest}
The authors declare no conflict of interest.


\end{document}